\documentclass[11pt]{amsart}
\usepackage{inputenc,euscript,amssymb,geometry}
\usepackage{graphicx}
\usepackage{latexsym}
\usepackage{mathtools}
\usepackage{amsbsy,amsfonts,amscd}

\usepackage{hyperref}
\usepackage{color}

\begin{document}
\newtheorem{lem}{Lemma}[section]
\newtheorem{prop}{Proposition}[section]
\newtheorem{cor}{Corollary}[section]
\newtheorem{remark}{Remark}[section]
\numberwithin{equation}{section}
\newtheorem{thm}{Theorem}[section]
\theoremstyle{remark}
\newtheorem{example}{Example}[section]
\newtheorem*{ack}{Acknowledgment}
\theoremstyle{definition}
\newtheorem{definition}{Definition}[section]
\newtheorem*{notation}{Notation}

\newenvironment{Abstract}
{\begin{center}\textbf{\footnotesize{Abstract}}%
\end{center} \begin{quote}\begin{footnotesize}}
{\end{footnotesize}\end{quote}\bigskip}
\newenvironment{nome}
{\begin{center}\textbf{{}}%
\end{center} \begin{quote}\end{quote}\bigskip}

\newcommand{\triple}[1]{{|\!|\!|#1|\!|\!|}}
\newcommand{\xx}{\langle x\rangle}
\newcommand{\ep}{\varepsilon}
\newcommand{\al}{\alpha}
\newcommand{\be}{\beta}
\newcommand{\de}{\partial}
\newcommand{\la}{\lambda}
\newcommand{\La}{\Lambda}
\newcommand{\ga}{\gamma}
\newcommand{\del}{\delta}
\newcommand{\Del}{\Delta}
\newcommand{\sig}{\sigma}
\newcommand{\ome}{\Omega^n}
\newcommand{\Ome}{\Omega^n}
\newcommand{\C}{{\mathbb C}}
\newcommand{\N}{{\mathbb N}}
\newcommand{\Z}{{\mathbb Z}}
\newcommand{\R}{{\mathbb R}}
\newcommand{\T}{{\mathbb T}}
\newcommand{\Rn}{{\mathbb R}^{n}}
\newcommand{\Rnu}{{\mathbb R}^{n+1}_{+}}
\newcommand{\Cn}{{\mathbb C}^{n}}
\newcommand{\spt}{\,\mathrm{supp}\,}
\newcommand{\Lin}{\mathcal{L}}
\newcommand{\SSS}{\mathcal{S}}
\newcommand{\F}{\mathcal{F}}
\newcommand{\xxi}{\langle\xi\rangle}
\newcommand{\eei}{\langle\eta\rangle}
\newcommand{\xei}{\langle\xi-\eta\rangle}
\newcommand{\yy}{\langle y\rangle}
\newcommand{\dint}{\int\!\!\int}
\newcommand{\hatp}{\widehat\psi}
\renewcommand{\Re}{\;\mathrm{Re}\;}
\renewcommand{\Im}{\;\mathrm{Im}\;}

\title[Scattering for NLS with a delta potential]
{Scattering for NLS with a delta potential}

\author[Valeria Banica]{Valeria Banica}
\author[Nicola Visciglia]{Nicola Visciglia}
\address[V. Banica]{D\'epartement de Math\'ematiques, Universit\'e d'Evry, LaMME (UMR 8071), 23 Bd. de France, 91037 Evry\\ France} 
\email{Valeria.Banica@univ-evry.fr}
\address{Dipartimento di Matematica Universit\`a Degli Studi di Pisa,
Largo Bruno Pontecorvo 5 I - 56127 Pisa, Italy}
\email{viscigli@dm.unipi.it}
\maketitle
\begin{abstract}
We prove $H^1$ scattering for defocusing NLS with a delta potential
and mass-supercritical nonlinearity, hence extending
in an inhomogeneous setting the classical $1-D$ scattering results
first proved by Nakanishi in the translation invariant case.
\end{abstract}
\tableofcontents
\section{Introduction}
We consider the defocusing Schr\"odinger equation on the line with a delta potential of strength $q>0$:
\begin{equation}\label{NLS}
\left\{\begin{array}{c}i\partial_tu-H_q u - u |u|^\alpha =0,
\quad \alpha>4,\\u_{\restriction t=0}=\varphi
\end{array}\right.
\end{equation}
where $H_q=-\frac 12 \partial_{xx} u+q\delta$ is a self-adjoint operator on the domain
$$\mathcal D(H_q)=\{f\in\mathcal C(\mathbb R)\cap H^2(\mathbb R\setminus\{0\}), f'(0^+)-f'(0^-)=2q f(0)\}.$$ 
The quadratic form associated with $H_q$ is $\frac 14\|\partial_xf\|_{L^2}^2+\frac q2|f(0)|^2$, on the energy space $H^1(\R)$ (see for instance Adami-Noja \cite{AdNo09}). 
We underline that in the case $q=0$ the operator $H_0$ is the classical Laplace operator $
-\frac 12 \partial_{xx}$ on the domain $H^2(\R)$ and \eqref{NLS} reduces to
\begin{equation}\label{NLSconstant}
\left\{\begin{array}{c}i\partial_tu+\frac 12 \partial_{xx} u - u |u|^\alpha =0,\\u_{\restriction t=0}=\varphi.\end{array}\right.
\end{equation}
The operator $H_q$ describes a $\delta-$interaction of strength $q$ centered at $x=0$. On the one hand, this kind of interaction, known also as Fermi pseudopotential, give rise to many currently used models in physiscs. 
We refer to the monograph of Albeverio-Gesztesy-H{\o}egh-Krohn-Holden \cite{AlCo}. 
We remark also that \eqref{NLS} is the simplest case of the nonlinear Schr\"odinger equation posed on a metric graph with delta-conditions at the vertices, namely when the graph has only one vertex and two edges. On the other hand, the qualitative properties of the solutions of the nonlinear Schr\"odinger equation with a potential is a subject of current interest. First a series of studies dealt with the dispersive properties of the perturbed linear operator in $1-D$
(Christ-Kiselev \cite{christkiselev},
D'Ancona-Fanelli \cite{danconafanelli}, Goldberg-Schlag \cite{GoSc04}, Weder \cite{weder}, Yajima \cite{yajima} to quote a few of them...). 
Also a huge literature has been developed around
the corresponding perturbed nonlinear equations in $1-D$, to quote the most recent results
we mention Carles \cite{Ca14}, Cuccagna-Georgiev-Visciglia \cite{CuGeVi14},
Germain-Hani-Walsh
\cite{germainhaniwalsh}, and all the references therein.\\
Let us recall now the facts known about \eqref{NLS}. In the repulsive case $q\geq 0$, the free solutions can be computed explicitly (see Gaveau-Schulman \cite{GaSc86}), yielding the classical dispersion estimate $\|e^{-itH_q}f\|_{L^\infty}\leq Ct^{-\frac 12}\|f\|_{L^1}$ and therefore classical Strichartz estimates. These estimates remain valid in the attractive case $q<0$, up to projecting outside the discrete spectrum, composed by the unique eigenvalue $-\frac{q^2}{4}$ associated with the  eigenvector $u_q(x)=\sqrt{\frac{|q|}2}e^{q|x|}$ (see Adami-Sacchetti \cite{AdSa05}). 
The nonlinear problem \eqref{NLS} is therefore globally well-posed in $H^1$ and that the mass $\int_\mathbb R |u(t,x)|^2 dx$ and the energy
$$E(u(t))=\frac 14\int_{\mathbb R}|\partial_x u(t,x)|^2 dx
+\frac {q}2|u(t,0)|^2 +\frac 1{\alpha+2}\int_{\mathbb R}| u(t,x)|^{\alpha+2}dx$$
are two conserved in time quantities. Let us mention also that \eqref{NLS} in the focusing cubic case, i.e. opposite sign in front of the nonlinearity, slow and fast solitons evolutions have been studied in a series of papers Goodman-Holmes-Weinstein \cite{GoHoWe04}, Holmer-Zworsky \cite{HoZw07}, Holmer-Marzuola-Zworsky \cite{HoMaZw07},\cite{HoMaZw07bis}, Datchev-Holmer \cite{DaHo09}. Also, stability results for bound states were obtained in 
Adami-Noja-Visciglia \cite{AdamiNojaVisciglia}, Fukuizumi-Ohta-Ozawa \cite{FuOhOz08}, Le Coz-Fukuizumi-Fibish-Ksherim-Sivan \cite{LeFuFiKsSi08}, Holmer-Zworsky \cite{HoZw09}, Deift-Park \cite{DePe11}. 
Finally, let us note that  in both cubic cases with repulsive potential small data long-range wave operators in $L^2$ were recently proved by Segata \cite{Se14}.

Our main contribution is the proof of the asymptotic completeness for \eqref{NLS} in $H^1(\mathbb R)$ for $\alpha>4$, in the repulsive case $q>0$. 
We recall that in the case $q=0$ this result was first proved by Nakanishi in \cite{Na99} by using a weighted in space and time Morawetz inequality.
New proofs have been provided via interaction Morawetz estimates
in the papers Colliander-Holmer-Visan-Zhang \cite{CoHoViZh08}, Colliander-Grillakis-Tzirakis \cite{CoGrTz09}, ~Planchon-Vega \cite{PlVe09}.\\
Next we state our result.
\begin{thm}\label{main}
Let $\varphi\in H^1(\R)$ be given and $u(t, x)\in {\mathcal C}(\R; H^1(\R))$
be the unique global solution to \eqref{NLS} with $q>0$ and $\alpha>4$. Then there exist
$\varphi_\pm \in H^1(\R)$ such that
\begin{equation}\label{scfi}\|e^{-itH_q} \varphi_{\pm} - u(t, x)\|_{H^1(\R)} 
\overset{ t\rightarrow \pm 
\infty } \longrightarrow 0.\end{equation}
\end{thm}


In the sequel we use the following compact notation  for any $\varphi\in H^1(\R)$:
$$Sc(\varphi) \hbox{ occurs } \iff \eqref{scfi} \hbox{ is true for suitable } \varphi_{\pm}\in H^1(\R)$$ where
$u(t, x)\in {\mathcal C} (\R;H^1(\R))$ is the unique global solution to \eqref{NLS}.
\\
\\
The proof is heavily based on the concentration-compactness/rigidity
technique first introduced by Kenig-Merle in \cite{KeMe06} and borrows arguments from Duyckaerts-Holmer-Roudenko \cite{DuHoRo08} and Fang-Xie-Cazenave \cite{FaXiCa11}. 
In fact the main difficulty in our context is the lack of translation invariance
of the equation, due to the delta interaction. The same difficulty 
appears in the paper \cite{Ho14} by Hong where he considers NLS in 3-D with a 
potential type perturbation. In this case the lack of homogeneity is 
solved thanks to the choice of a suitable Strichartz couple 
that allows to prove smallness of a suitable reminder. 
This technique seems to be non useful in the 1-D case because of a numerology problem. To give an idea of the main
difference between Strichartz estimates
in 1D and 3D, recall that in 1-D Strichartz estimates are 
far from reaching the $L^2$ time summability that is available in 3-D.\\
Moreover the delta interaction
is a singular perturbation and hence the profile decomposition proof, 
as well as the construction of the minimal element,
cannot be given in the perturbative spirit as in Hong proof, where linear scattering is at hand. We believe
that the proof of the profile decomposition associated with a delta type interaction,
given along this paper, has its own interest.
In particular it does not rely on the corresponding profile decomposition
available in the free case. 
\\
\\
In the focusing cases $q<0$ or an opposite sign in front of the nonlinearity in \eqref{NLS} the above arguments can be used to prove scattering up to the natural threshold, given in terms of ground states, between global existence and blow-up.

\begin{notation}
We shall use the following notations without any further comments:
$$L^p=L^p(\R), H^s=H^s(\R), L^pL^q=L^p (\R; L^q(\R),
{\mathcal C} H^s={\mathcal C} (\R; H^s(\R)).$$
We also denote by $(.,.)$ the usual $L^2$ scalar product and 
by $(.,.)_{H^1}$ the scalar product in $H^1$, i.e.
$(f, g)_{H^1}=\int_\R f(x) \bar g(x) dx+\int_{\R} f'(x) \bar g'(x) dx$.
We denote by $\tau_x$ the translation operator, i.e. $\tau_x f(y)=f(y-x)$.
Given a sequence $(x_n)_{n\in\mathbb N}$ we denote by $x_n \overset{ X} \rightarrow x$
and $x_n \overset{ X } \rightharpoonup x$
respectively the strong and weak convergence in the topology of $X$ 
as $n\rightarrow \infty$.
\end{notation}

{\bf{Acknowledgements:}} The authors are grateful to Prof. Riccardo
Adami for interesting discussions. V.B. is partially
supported by the French ANR project "SchEq" ANR-12-JS01-0005-01,
N.V. is supported by FIRB project Dinamiche Dispersive.

\section{Profile decomposition}\label{profi}

\subsection{The general case}

The aim of this section is the proof of profile decomposition associated with 
a general family
of propagators $e^{-itA}$.
From now on $A$ will denote a self-adjoint operator 
$$A:L^2\supset D(A)\ni u\mapsto A u\in L^2$$
that satisfies suitable assumptions.
More precisely we assume the following:
\begin{itemize}\item
there exist $c, C>0$ such that
\begin{equation}\label{equiv}
c\|u\|_{H^1}^2\leq (Au,u)+\|u\|_{L^2}^2\leq C \|u\|_{H^1}^2, \quad \forall u\in D(A);
\end{equation}
\item let $B:D(A)\times D(A)\ni (f, g)\mapsto B(f, g)\in \C$ be
defined as follows:
\begin{equation}\label{structbil}
(Au, v)= (u, v)_{H^1}+ B(u,v),  \quad \forall u,v\in D(A)\times D(A),
\end{equation}
then
\begin{equation}\label{structbil1}
B(\tau_{x_n} \psi, \tau_{x_n}h_n)\overset {n\rightarrow \infty} \longrightarrow
0, \quad \forall \psi\in H^1
\end{equation}
provided that:
\begin{align*}
\hbox{ either } x_n\overset{n\rightarrow \infty }
\longrightarrow \pm \infty, \quad \sup_n \|h_n\|_{H^1}<\infty,\\
\nonumber
\hbox{ or } x_n\overset{n\rightarrow \infty } \longrightarrow \bar x\in \R,
\quad h_n\overset{ H^1} \rightharpoonup 0;
\end{align*}

\item let $(t^n)_{n\in\mathbb N}$, $(x^n)_{n\in\mathbb N}$ be sequences
of real numbers, then we have the following implications:
\begin{equation}\label{abstr2}
t^n\overset{n\rightarrow \infty}
 \longrightarrow \pm \infty \Longrightarrow 
 \|e^{it^nA} \tau_{x^n} \psi\|_{L^p} \overset{n\rightarrow \infty
 }\longrightarrow 0,
 \quad 2<p<\infty,
 \quad \forall \psi\in H^1;
\end{equation}
\begin{align}\label{hyppa}
 &t^n\overset{n\rightarrow \infty}
 \longrightarrow \bar t\in \R, \quad x^n\overset {n\rightarrow \infty}
\longrightarrow \pm \infty \Longrightarrow
\\\nonumber   \forall \psi \in H^1\quad & \exists \tilde\psi \in H^1,\quad
\|\tau_{-x^n}e^{it^n A} \tau_{x^n}\psi -\tilde\psi\|_{H^1}\overset{ n\rightarrow \infty}
\longrightarrow 0 ;\end{align}
\begin{equation}\label{hyppb}
 t^n\overset{n\rightarrow \infty}
 \longrightarrow \bar t\in \R, \quad x^n\overset {n\rightarrow \infty}
\longrightarrow  \bar x\in \R
\Longrightarrow  
\|e^{it^n A} \tau_{x^n}\psi -e^{i\bar t A} \tau_{\bar x}\psi\|_{H^1}\overset{ n\rightarrow \infty}
\longrightarrow 0 , \quad \forall \psi \in H^1.\end{equation}
\end{itemize}



We can now state the main result of this section.
\begin{thm}\label{profile}
Let $(u_n)_{n\in\mathbb N}$ be a sequence bounded in $H^1$
and let $A$ be a self-adjoint operator that satisfies
\eqref{equiv}, \eqref{structbil}, \eqref{structbil1}, 
\eqref{abstr2}, \eqref{hyppa} and \eqref{hyppb}.
Then, up to subsequence, we can write
$$u_n=\sum_{j=1}^{J} e^{it_{j}^nA} \tau_{x_{j}^n}\psi_j + R_n^J, \quad
\forall J\in \N$$
where,
$$t_{j}^n\in \R, \quad x_{j}^n\in \R,\quad  \psi_j\in H^1$$
are such that: 
\begin{itemize}
\item for any fixed $j$ we have:
\begin{align}\label{localiz}
&\hbox{ either } t_{j}^{n}=0, \quad \forall n, \quad \mbox { or } 
\quad t^{n}_{j} 
\overset{ n\rightarrow \pm \infty } \longrightarrow \pm\infty,\\\nonumber
&\hbox{ either } \quad x_{j}^{n}=0, \quad \forall n, \quad \mbox { or }  \quad x^{n}_{j}
\overset{ n\rightarrow \infty } \longrightarrow \pm\infty;\end{align}
\item orthogonality of the parameters:
\begin{equation}\label{orth}
|t_{j}^{n}-t_{k}^{n}|+ |x_{j}^{n}-x_{k}^{n}|\overset{
n\rightarrow \infty}
\longrightarrow \infty, \quad \forall j\neq k;\end{equation}
\item smallness of the reminder:
\begin{equation}\label{reminder} \forall \epsilon>0 \hbox{ } \exists J=J(\epsilon)\in \N
\hbox{ such that }  \limsup_{n\rightarrow \infty} \|e^{-itA} R_n^J\|_{L^\infty L^\infty}\leq \epsilon;
\end{equation}
\item orthogonality in Hilbert norms:
\begin{equation}\label{orthogL2}
\|u_n\|_{L^2}^2=\sum_{j=1}^J \|\psi_j\|_{L^2}^2+\|R_n^J\|_{L^2}^2
+o(1), \quad \forall J\in \N;
\end{equation}
\begin{equation}\label{orthog}
\|u_n\|_{H}^2=\sum_{j=1}^J \| \tau_{x_{j}^n} \psi_j\|_H^2+\|R_n^J\|_H^2
+ o(1), \quad \forall J\in \N,
\end{equation}
where $\|v\|_H^2=(Av,v)$.
\end{itemize}
Moreover we have
\begin{equation}\label{potential}
\|u_n\|_{L^p}^p=\sum_{j=1}^J \| e^{it_{j}^nA} \tau_{x_{j}^n} \psi_j\|_{L^p}^p+\|R_n^J\|_{L^p}^p
+ o(1), \quad p\in (2, \infty), \quad \forall J\in \N,
\end{equation}
and in particular
\begin{equation}\label{orthogenergy}
E(u_n)= \sum_{j=1}^J E(e^{it_{j}^nA} \tau_{x_{j}^n} \psi_j)+ 
E (R_n^J) +o(1), \quad \forall J\in \N,
\end{equation}
where $E(u)= \frac 12 \|u\|_{H}^2 + \frac 1{\alpha+2}\|u\|_{L^{\alpha+2}}^{\alpha+2}$.
\end{thm}

In order to prove the theorem, we need first the following  lemma,
where  we implicitly assume the same assumptions as in Theorem
\ref{profile}.

\begin{lem} Let $(h_n)_{n\in\mathbb N}$ be bounded in $H^1$ and $
(t^n , t_1^n , t_2^n,
t^n, x_1^n, x_2^n)_{n\in\mathbb N}
$ be sequences of real numbers. Then
we have the following implications:
\begin{equation}\label{abstract}
h_n \overset{H^1} \rightharpoonup 0, \quad \tau_{-x^n_2}e^{i(t^n_2- t^{n}_1) A}\tau_{x^{n}_1} h_n \overset{H^1}
\rightharpoonup \psi\neq 0
\Longrightarrow 
|t^{n}_1-t^{n}_2|+ |x^{n}_1-x^{n}_2|\overset{n\rightarrow \infty}
\longrightarrow \infty;
\end{equation}
\begin{equation}\label{hyppa3}
h_n \overset{H^1} \rightharpoonup 0, \quad
t^n\overset{n\rightarrow \infty}
 \longrightarrow \bar t\in \R, \quad x^n \overset{ n\rightarrow \infty } \longrightarrow \pm \infty
 \Longrightarrow \tau_{-x^n}e^{it^n A} \tau_{x^n} h_n \overset{H^1}
\rightharpoonup 0.
 \end{equation}

\end{lem}

\begin{proof}
For $\psi\in H^1$ we get 
$$(\tau_{-x^n_2}e^{i(t^n_2- t^{n}_1) A}\tau_{x^{n}_1} h_n,\psi)=
(h_n, \tau_{-x^{n}_1} e^{-i(t^n_2- t^{n}_1) A}\tau_{x^n_2}\psi),$$
where $(.,.)$ denotes the $L^2$-scalar product. Hence
by \eqref{hyppa} we have $\tau_{-x^n_2}e^{i(t^n_2- t^{n}_1) A}\tau_{x^{n}_1} h_n\overset{L^2}
\rightharpoonup 0$, and up to subsequence we obtain \eqref{hyppa3}.

Concerning \eqref{abstract} it is equivalent to prove that
\begin{align*}
h_n \overset{H^1} \rightharpoonup 0&, 
\quad s^n \overset{n\rightarrow \infty} \longrightarrow \bar s \in \R,
\quad y^n-z^n\overset{n\rightarrow \infty} \longrightarrow \bar z \in \R \\\nonumber&\Longrightarrow 
\tau_{-z^n}e^{is^nA} \tau_{y_n} h_n \overset{H^1} \rightharpoonup 0.
\end{align*}
This fact is equivalent to
$$\tau_{-z^n}e^{is^nA} \tau_{z_n} g_n \overset{H^1} \rightharpoonup 0$$
where $g_n=\tau_{\bar z} h_n\overset{H^1} \rightharpoonup 0$.
Hence we conclude by \eqref{hyppa3} in the case $z^n \overset{n\rightarrow \infty} \longrightarrow\pm  \infty$; in the case 
$z^n\overset{n\rightarrow \infty} \longrightarrow \bar z\in \R$ we conclude 
by
the strong convergence of the sequence of operators
$(\tau_{-z^n}e^{is^nA} \tau_{z_n})_{n\in \N}$ to the operator 
$\tau_{-z^*}e^{i\bar sA} \tau_{\bar z}$.

\end{proof}
\begin{lem}\label{maximal}
Let $(v_n)_{n\in\mathbb N}$ be bounded in $H^1$ 
and
$$\Gamma(v_n)=\{w\in L^2 \quad | \quad \exists (x_k)_{k\in \N} \in \R,
(n_k)_{k\in \N}\in \N \hbox{ with }  n_k\nearrow \infty \hbox{ s. t. }
\tau_{x_k} (v_{n_k}) \overset{L^2}\rightharpoonup w\}.$$
Then there exist $M=M( 
\sup_n \|v_n\|_{H^1})>0$, such that:
\begin{equation}\label{prop}\limsup_{n\rightarrow \infty}
\|v_n\|_{L^\infty}\leq M (\gamma(v_n))^{1/3},
\hbox { where }
\gamma(v_n)=\sup_{w\in \Gamma (v_n)} \|w\|_{L^2}.\end{equation}
\end{lem}

\begin{proof}
We introduce the Fourier multipliers
$\chi_R(|D|)$ and $\tilde \chi_R (|D|)$ 
associated with the function
$\chi(\frac \xi R)$ and $1-\chi(\frac \xi R)$
where $$\chi(\xi)\in C^\infty(\R), \quad
\chi(x)=1 \hbox{ for } \quad |x|<1, \quad \chi(x)=0
\hbox{ for } \quad |x|>2.$$
Recall  that
$H^{3/4}\subset L^\infty$ and hence
\begin{equation}\label{8ba}\|\tilde \chi_R (|D|)v_n\|_{L^\infty}
\leq C  R^{-\frac 14} \|v_n\|_{H^1}.\end{equation}
In order to estimate
$\|\chi_R (|D|)v_n\|_{L^\infty}$
we select $\{y_n\}\in \R$ such that
\begin{equation}\label{9ba}\| \chi_R (|D|)v_n\|_{L^\infty} \leq 2 |\chi_R (|D|)v_n(y_n)|.
\end{equation}
Notice that
$$|\chi_R (|D|)v_n (y_n)|=R |\int \eta(R x) v_n(x-y_n) dx| $$
where $\hat \chi=\eta$.
Moreover for every subsequence $(n_k)_{k\in \N}$ 
we can select another subsequence $\{n_{k_h}\}$ such that
$v_{n_{k_h}}(x-y_{n_{h_k}})\overset{L^2} \rightharpoonup w\in \Gamma(v_n)$
and hence
$$\limsup_{h\rightarrow \infty} 
|\chi_R (|D|)v_{n_{k_h}} (y_{n_{k_h}})|=
R |\int \eta(R x) w dx| \leq C R \|\eta(R x)\|_{L^2}\|w\|_{L^2}
\leq C\sqrt R \,\gamma(v_n)$$
which implies
$$\limsup_{n\rightarrow \infty} 
|\chi_R (|D|)v_{n} (y_{n})|
\leq C\sqrt R\, \gamma(v_n).$$
By combining this estimate with \eqref{8ba} and \eqref{9ba}
we get:
$$\limsup_{n\rightarrow \infty}
\|v_n\|_{L^\infty}\leq C R^{-\frac 14}\sup_n \|v_n\|_{H^1} +C \sqrt R\, \gamma(v_n)$$
and we conclude by choosing $R=C(\sup_n \|v_n\|_{H^1})\, (\gamma(v_n))^{-\frac 43}$.

\end{proof}

\begin{lem}\label{lemprof}
Let $(v_n)_{n\in \N}$ be bounded in $H^1$.
Then up to subsequence there exist $\psi\in H^1$, $
(x^n)_{n\in \N}, (t^n)_{n\in \N}$ sequences of real numbers
and $L=L(\sup_n \|v_n\|_{H^1})>0$ such that,
\begin{equation}\label{mal}\tau_{-x^n} (e^{-it^n A} v_n)=\psi+ W_n\end{equation}
where:
\begin{align}
\label{10ba} &W_n \overset{H^1}\rightharpoonup 0;\\
\label{1ba}
&\limsup_{n\rightarrow \infty}
\|e^{-it A} v_n\|_{L^\infty L^\infty} \leq L \|\psi\|_{L^2}^{1/3};\\
\label{27ba}&\|v_n\|_{L^2}^2 =\|\psi\|_{L^2}^2+ \|W_n\|_{L^2}^2+o(1);\\
\label{299ba}&\|v_n\|_{H}^2 =\|\tau_{x^n}\psi\|_{H}^2+ \|\tau_{x^n}W_n\|_{H}^2+o(1),
\hbox{ where } \|v\|_H^2=(Av, v);\\
\label{lebe}&\|v_n\|_{L^p}^p= \| e^{it^{n} A} \tau_{x^{n}} \psi \|_{L^p}^p+
\|e^{it^{n} A} \tau_{x^{n}} W_n\|_{L^p}^p
+ o(1),\quad\forall p\in(2,\infty).
\end{align}
Moreover we can assume (up to subsequence): 
\begin{align}\label{localizbis}
&\hbox{ either } t^{n}=0, \quad \forall n, \quad \mbox { or } 
\quad t^{n} 
\overset{ n\rightarrow \pm \infty } \longrightarrow \pm\infty,\\\nonumber
&\hbox{ either } \quad x^{n}=0, \quad \forall n, \quad \mbox { or }  \quad x^{n}
\overset{ n\rightarrow \infty } \longrightarrow \pm\infty.\end{align}

\end{lem}

\begin{proof} 
First we give the definition of $\psi, (W_n)_{n\in \N}, (t^n)_{n\in \N},  
(x^n)_{n\in \N}$ in \eqref{mal}. 
Let $(t^n)_{n\in \N}$ be a sequence of real numbers such that
\begin{equation}\label{3ba}\|e^{-it^n A} v_n\|_{L^\infty}>\frac 12 \|e^{-it A} v_n\|_{L^\infty L^\infty}.
\end{equation}
Note that we get the boundedness
of $(e^{-itA}v_n)_{n\in \N}$ in $H^1$ by using \eqref{equiv} 
and the assumption on the 
boundedness of $(v_n)_{n\in \N}$ in $H^1$.
Following the notations of Lemma \ref{maximal}
we introduce
$\Gamma (e^{-it^n A} v_n)\subset L^2$ and also $\gamma (e^{-it^n A} v_n)\in [0, \infty)$.
Then up to subsequence 
we get the existence of $(x^n)_{n\in \N}\subset \R$ and $\psi\in H^1$ such that
\begin{equation}\label{4ba}\tau_{-x^n}(e^{-it^n A} v_n) 
\overset{H^1} \rightharpoonup \psi
\end{equation} and
\begin{equation}\label{5ba}\|\psi\|_{L^2}\geq \frac 12 \gamma(e^{-it^n A} v_n).
\end{equation}
On the other hand by combining Lemma \ref{maximal} with \eqref{5ba}
we get 
$$\limsup_{n\rightarrow \infty}\|e^{-it^n A} v_n\|_{L^\infty}
\leq 2^{1/3} M \|\psi\|_{L^2}^{1/3}.$$
By combining this estimate with \eqref{3ba} we get \eqref{1ba}.\\
\\The proof of \eqref{10ba} follows by \eqref{4ba} together with the definition of $W_n$
in \eqref{mal}.\\
\\
To prove \eqref{27ba} we combine
\eqref{mal}, \eqref{10ba} and the Hilbert structure of $L^2$
in order to get
$$\|v_n\|_{L^2}^2 =\|\tau_{-x^n} (e^{-it^n A} v_n)\|_{L^2}^2 =\|\psi\|_{L^2}^2+ \|W_n\|_{L^2}^2+o(1),$$
where we used that $\tau_{x} e^{-it A}$ is an isometry in $L^2$ for every $(t, x)$.\\
\\
Next we prove \eqref{299ba}.
By \eqref{mal}
we get
$$v_n=e^{it^n A}\tau_{x^n} \psi+ e^{it^n A}\tau_{x^n}W_n$$
and hence \eqref{299ba} follows
provided that
$$(e^{it^n A}\tau_{x^n} \psi, e^{it^n A}\tau_{x^n}W_n)_H
\overset{n\rightarrow \infty}
\longrightarrow 0.$$
Notice that we have
$$(e^{it^n A}\tau_{x^n} \psi, e^{it^n A}\tau_{x^n}W_n)_H=
(\tau_{x^n} \psi, \tau_{x^n}W_n)_H=(\tau_{x^n} \psi, \tau_{x^n}W_n)_{H^1}
+ B(\tau_{x^n} \psi, \tau_{x^n}W_n)$$
where we used \eqref{structbil}. Therefore
$$(e^{it^n A}\tau_{x^n} \psi, e^{it^n A}\tau_{x^n}W_n)_H=(\psi, W_n)_{H^1}+ B(\tau_{x^n} \psi, \tau_{x^n}W_n).$$
Up to subsequence we have either
$x^n\overset{ n\rightarrow \infty } \longrightarrow \pm \infty$ or  $x^n
\overset{ n\rightarrow \infty } \longrightarrow \bar x\in\R$, and 
in both cases we conclude by
\eqref{structbil1}.
\\
\\
Next we prove \eqref{lebe}. We can assume that,
up to subsequence, we are in one of the following cases:
\\
\\
{\em First case: $t^n\overset{n\rightarrow \infty} \longrightarrow \pm \infty$.}
\\
\\
Since
$$v_n= e^{it^n A} \tau_{x^n} \psi + e^{it^n A} \tau_{x^n} W_n$$
and $W_n$ is uniformly bounded in $H^1$,  we conclude by assumption
\eqref{abstr2}.
\\
\\
{\em Second case: $t^n\overset{n\rightarrow \infty}
\longrightarrow \bar t\in \R, \quad x^n\overset{n\rightarrow \infty}
\longrightarrow \bar x \in \R$.}
\\
\\
Notice that we have
\begin{equation}\label{procbl}v_n- e^{it^n A} \tau_{x^n}\psi=
e^{it^n A} \tau_{x^n} W_n 
\overset{ n\rightarrow \infty } \longrightarrow 0 \hbox{ a.e. } x\in\R.\end{equation}
The last property follows by $$(e^{it^n A} \tau_{x^n} W_n ,\varphi)_{L^2}
= (W_n, \tau_{-x^n} e^{-it^n A} \varphi)_{L^2}=
(W_n, \tau_{-\bar x} e^{-i\bar t A} \varphi)_{L^2}+ o(1)
=o(1)$$ that implies $e^{it^n A} \tau_{x^n} W_n \overset{L^2} \rightharpoonup 0$.
Hence by $H^1$-boundedness
$e^{it^n A} \tau_{x^n} W_n \overset{H^1} \rightharpoonup 0$
(up to subsequence). We conclude
by Rellich Theorem the convergence in $L^2_{loc}$ which in turn implies pointwise convergence.
By combining \eqref{hyppb} with \eqref{procbl} we get, up to subsequence,
$$v_n-e^{i\bar t A} \tau_{\bar x}\psi= e^{it^n A} \tau_{x^n} W_n + h_n(x)
\overset{ n\rightarrow \infty } \longrightarrow 0 \hbox{ a.e. } x, \quad \|
h_n\|_{L^p}\overset{ n\rightarrow \infty } \longrightarrow 0,$$
and hence by the Br\'ezis-Lieb Lemma (see \cite{BrezisLieb}) we get
$$\| e^{i t^n A} \tau_{x^n} W_n\|_{L^p}^p= \|v_n\|_{L^p}^p
- \|e^{i\bar t A} \tau_{\bar x} \psi\|_{L^p}^p+o(1)$$
$$=\|v_n\|_{L^p}^p
- \|e^{i t^n A} \tau_{x^n} \psi\|_{L^p}^p+o(1).$$
\\
\\
{\em Third case: $t^n\overset{n\rightarrow \infty}\longrightarrow \bar t\in \R,
\quad x^n\overset{n\rightarrow \infty}\longrightarrow \pm \infty$.}
\\
\\
We have by \eqref{mal} and \eqref{hyppa3}
\begin{equation}\label{hypva}\tau_{-x^n} v_n- \tau_{-x^n}e^{it^n A} \tau_{x^n}\psi=
\tau_{-x^n}e^{it^n A} \tau_{x^n} W_n \overset{H^1} \rightharpoonup 0,
\end{equation}
so as before we obtain pointwise convergence towards zero. Moreover, by \eqref{hyppa} we get $\tilde\psi\in H^1$ such that
\begin{equation}\label{dot}\tau_{-x^n}e^{it^n A} \tau_{x^n}\psi
\overset{ H^1} \longrightarrow  \tilde\psi.
\end{equation}
By combining the pointwise convergence and \eqref{dot} with 
the Br\'ezis-Lieb Lemma and with the translation invariance
of the $L^p$ norm we get 
\begin{align*}&\|e^{it^n A} \tau_{x^n} W_n\|_{L^p}^p=\|v_n\|_{L^p}^p- \|
\tilde\psi\|_{L^p}^p+ o(1)
\\\nonumber = \|v_n\|_{L^p}^p- \|\tau_{-x^n} e^{it^n A} \tau_{x^n}\psi\|_{L^p}+ o(1)
&= \|v_n\|_{L^p}^p- \|e^{it^n A} \tau_{x^n}\psi\|_{L^p}+ o(1).
\end{align*}
\\
\\
Finally we focus on \eqref{localizbis}.
First we consider the case $(t^n)_{n\in \N}$ is bounded, when we get
up to subsequence
$t^n \overset{ n\rightarrow \infty} \longrightarrow \bar t\in \R$.
Next we consider three cases (that can occur up to subsequence).
\\
\\
{\em First case: $t^n \overset{ n\rightarrow \infty} \longrightarrow \bar t\in \R, 
\quad x^n\overset{ n\rightarrow \infty } \longrightarrow \pm \infty$.}
\\
\\
In this case we claim that we can write a new identity of the type
\eqref{mal} by replacing
$\psi$ by another $\tilde \psi$, the sequence $(W_n)_{n\in \N}$ by another
sequence $(\tilde W_n)_{n\in \N}$ and the parameters
$(t^n, x^n)$ by $(0, x^n)$, where $\tilde W_n\overset{H^1} \rightharpoonup 0$.\\
Then the proof of \eqref{27ba}, \eqref{299ba}, \eqref{lebe}
(where we replace $\psi$ by $\tilde \psi$
and $W_n$ by $\tilde W_n$) follows as above. 
Moreover the proof of \eqref{1ba}, with $\psi$ replaced by $\tilde \psi$,
is trivial and follows by the fact that by the construction below we get
$\|\tilde \psi\|_{L^2}=\|\psi\|_{L^2}$.
\\
\\
Recall that thanks to \eqref{hyppa} we get the existence of $\tilde\psi\in H^1$ such that
$$\tau_{-x^n} e^{it^n A} \tau_{x^n} \psi\overset{H^1}\longrightarrow \tilde\psi.$$
Hence, in view of \eqref{mal} (with the parameters $(t^n, x^n)$
and functions $\psi, (W_n)_{n\in \N}$ constructed above)
$$v_n=\tau_{x^n} \tilde \psi 
+e^{i t^n A} \tau_{x^n}
W_n + r_n(x), \quad \|r_n\|_{H^1}\overset{ n\rightarrow \infty } \longrightarrow 0.
$$
Hence we get the decomposition
$$\tau_{-x^n} v_n= \tilde \psi + \tilde W_n$$
where
$$\tilde W_n=\tau_{-x^n} e^{i t^n A} \tau_{x^n} W_n+ \tau_{-x^n} r_n.
$$
Notice that by \eqref{hyppa3} and $ \|r_n\|_{H^1}\overset{ n\rightarrow \infty } \longrightarrow 0$ we obtain
$\tilde W_n \overset{H^1}
\rightharpoonup 0$.
\\
\\
{\em Second case: $t^n \overset{ n\rightarrow \infty} \longrightarrow \bar t\in \R, 
\quad x^n\overset{n\rightarrow \infty}\longrightarrow \bar x\in \R$.}
\\
\\
We argue as in the previous case and we look for suitable $\tilde \psi $,
$(\tilde W_n)_{n\in \N}$, $(\tilde t^n, \tilde x^n)=(0, 0)$.
We select $\tilde \psi$ as follows:
$$\tilde \psi=e^{i\bar t A} \tau_{\bar x} \psi.$$
We then have
$$\tilde W_n=e^{it^n A} \tau_{x^n} \psi -\tau_{-\bar x} e^{i\bar t A} \tau_{\bar x} \psi
+e^{i t^n A} \tau_{x^n} W_n.
$$
Notice that in this case the property 
$\tilde W_n \overset{H^1}
\rightharpoonup 0$ follows by the fact that the sequence of operators $(e^{it^n A} \tau_{x^n})_{n\in \N}$ 
converge in strong topology sense to the operator
$e^{i\bar t A} \tau_{\bar x}$. 
We conclude as in the previous case.\\
\\
{\em Third case: $(t^n)_{n\in \N}$ is unbounded.}
\\
\\Up to subsequence, we can suppose $t^n \overset{ n\rightarrow \infty} \longrightarrow \pm\infty$. If $(x^n)_{n\in \N}$ is unbounded too, then up to subsequence $x^n \overset{ n\rightarrow \infty} \longrightarrow \pm\infty$ and we are done. In the case $(x^n)_{n\in \N}$ is bounded, then up to subsequence we can suppose $x^n \overset{ n\rightarrow \infty} \longrightarrow \bar x\in
\R$, and we choose $\tilde \psi=\tau_{\bar x} \psi,  (\tilde t^n, \tilde x^n)=(t^n, 0)$,
$\tilde W_n=(\tau_{x^n}  -\tau_{\bar x}) \psi
+ \tau_{x^n} W_n$.
We conclude as in the previous cases.
\end{proof}

\noindent{\bf Proof of Theorem \ref{profile}.}
We iterate several times Lemma \ref{lemprof}.\\
\\
{\em First step: construction of $\psi_1$.}\\
\\
By Lemma \ref{lemprof} we get
\begin{equation}\label{eq1}u_n=e^{it^n_1 A}(\tau_{x^n_1}\psi_1)+ R_n^1
\end{equation}
where $\psi=\psi_1$, $(t^n)_{n\in \N}=(t^n_1)_{n\in \N}, (x_1^n)_{n\in \N}
=(x^n)_{n\in \N}, (R_n^1)_{n\in \N}=(e^{it_1^n A}(\tau_{x_1^n}W_n^1))_{n\in \N}$, $(W_n^1)_{n\in \N} =(W_n)_{n\in \N}$ 
and $\psi, (t^n)_{n\in \N}, (x^n)_{n\in \N}, (W_n)_{n\in \N}$ 
are given by Lemma \ref{lemprof} for $(u_n)_{n\in \N}$
equal to $(v_n)_{n\in \N}$.
Moreover by \eqref{27ba} we get
\begin{equation}\label{ald1}\|u_n\|_{L^2}^2=\|\psi_1\|_{L^2}^2+ \|W_n^1\|_{L^2}^2
+o(1)=
\|\psi_1\|_{L^2}^2+ \|R_n^1\|_{L^2}^2 + o(1),\end{equation}
and by \eqref{10ba}
\begin{equation}\label{weak}\tau_{-x_1^n} (e^{-it_1^n A} R_n^1)=W_n^1 
\overset{H^1} \rightharpoonup 0.\end{equation}
The proof of \eqref{orthog} for $J=1$ follows by \eqref{299ba}, and we get
\begin{equation}\label{impza}\|u_n\|_H^2 =\|\tau_{x_1^n} \psi_1\|_{H}^2 + \|R_n^1\|_{H}^2 + o(1).\end{equation}
Moreover \eqref{potential} follows by \eqref{lebe}.
\\\\
{\em Second step: construction of $\psi_2$.}\\
\\\\
We apply again Lemma \ref{lemprof} to the sequence
$v_n=R_n^1=e^{it_{1}^n A}(\tau_{x^{n}_1}W_n^1)$
and we get
\begin{equation}\label{eq2}R_n^1= e^{it_2^n A}(\tau_{x_2^n}\psi_2)+ R_n^2
\end{equation}
where
$$R_n^2=e^{it_2^n A}(\tau_{x_2^n} W_n^2).$$
Moreover we get
\begin{equation}\label{ald2}\|R_n^1\|_{L^2}^2=\|\psi_2\|_{L^2}^2+ \|W_n^2\|_{L^2}^2+o(1),\end{equation}
and
\begin{equation}\label{brezlie}W_n^2 \overset{ H^1}\rightharpoonup 0.
\end{equation}
Summarizing by \eqref{eq1} and \eqref{eq2} we get 
$$u_n=e^{it_1^n A}(\tau_{x_1^n}\psi_1)+
e^{it_2^n A}(\tau_{x_2^n}\psi_2)+ R_n^2.$$
By combining \eqref{ald1} and \eqref{ald2} we get
$$\|u_n\|_{L^2}^2=\|\psi_1\|_{L^2}^2+\|\psi_2\|_{L^2}^2+ \|W_n^2\|_{L^2}^2
+o(1)=\|\psi_1\|_{L^2}^2+\|\psi_2\|_{L^2}^2+ \|R_n^2\|_{L^2}^2
+o(1).$$
Arguing as in the first step we can also prove
$$\|R_n^1\|_{H}^2= \|\tau_{x_n^2} \psi_2\|_H^2+\|R_n^2\|_H^2$$
and hence
by \eqref{impza}
$$\|u_n\|_H^2
=\|\tau_{x_1^n} \psi_1\|_{H}^2 + \|\tau_{x_2^n} \psi_2\|_{H}^2 +\|R_n^2\|_{H}^2 + o(1).
$$
Notice also that (for $J=2$) \eqref{potential} follows by \eqref{lebe}.\\
Next we prove that $(t_1^n)_{n\in \N}, (t_2^n)_{n\in \N}, 
(x_1^n)_{n\in \N}, (x_2^n)_{n\in \N}$ satisfy \eqref{orth}.
By combining \eqref{eq2} and \eqref{brezlie}
we get
$$\tau_{-x_2^n} e^{-it_2^n A}  e^{it_1^n A} \tau_{x_1^n} 
W_n^1=\tau_{-x_2^n} (e^{-it_2^n A}R_n^1)=\psi_2+ W_n^2\overset{ H^1} \rightharpoonup \psi_2,$$
hence either $\psi_2=0$ and we conclude the proof, or
$\psi_2\neq 0$ and we get \eqref{orth} by 
\eqref{abstract} and \eqref{weak}.
\\\\
{\em Third step: construction of $\psi_J$.}\\
\\\\
By iteration of the construction above we get
$$u_n=e^{it_1^n A}(\tau_{x_1^n}\psi_1)+...+
e^{it_J^n A}(\tau_{x_J^n}\psi_J)+ R_J^n$$
where $$R_J^n=e^{it_J^n A}(\tau_{x_J^n} W_n^J).$$
By repeating the computations above
we obtain \eqref{orthogL2}, \eqref{orthog} and \eqref{potential}. 
\\
Next we prove \eqref{reminder}.
Notice that 
by \eqref{orthogL2} and since $\sup_n \|u_n\|_{L^2}<\infty$ we get
that $\|\psi_J\|_{L^2}\overset{ J\rightarrow \infty } \longrightarrow 0$ 
and hence by \eqref{1ba} we get
$ \limsup_{n\rightarrow \infty} \|e^{it A} R_n^{J-1}\|_{L^\infty L^\infty}
\overset{ J\rightarrow \infty } \longrightarrow 0$.
\\
The proof of \eqref{orth} for generic $j,k$ 
is similar to the proof given in the second step  in the case $j=1, k=2$.
We skip the details.\\ 
Finally notice that \eqref{localiz} follows by \eqref{localizbis} 
in Lemma \ref{lemprof}.


\hfill$\Box$


\subsection{The case $A=H_q$.}

Along this section we verify the abstract assumptions 
\eqref{equiv}, \eqref{structbil}, \eqref{structbil1}, \eqref{abstr2}, \eqref{hyppa}, \eqref{hyppb}  
required on the operator $A$ along section \ref{profi}
in the specific case $$A=H_q=-\frac 12 \partial_x^2 + q\delta_0.$$
Notice that we get in this specific context
$B(f, g)= f(0) \bar g(0)$.
The verification of \eqref{equiv}, \eqref{structbil} 
and \eqref{structbil1} are straightforward
and follow by classical properties of the space $H^1$. Also the verification
of \eqref{hyppb} is trivial.\\\\
Next we shall verify \eqref{abstr2} and \eqref{hyppa} 
and we shall make extensively use of the following identity
(see Lemma 2.1 in \cite{HoMaZw07bis}) 
available for any initial datum $f\in L^1$ and supported in $(-\infty,0]$:
\begin{align}\label{ADelta}
e^{-itH_q}f(x)&=e^{-itH_0}f(x)\\\nonumber
+(e^{-itH_0}(f\star\rho_q))(x)\cdot 1_{x\geq 0}(x)
&+(e^{-it H_0}(f\star\rho_q))(-x) \cdot 1_{x\leq 0}(x),
\end{align}
where $\rho_q(x)=-q e^{q x} \cdot 1_{x\leq 0}(x)$.\\
\\
We check the validity of \eqref{hyppa}.
Since $[e^{-itH_0}, \tau_x]=0$
it is sufficient to prove that
\begin{equation*}\|\tau_{-x^n}e^{-it^nH_q}\tau_{x^n}\psi(x)-\tau_{-x^n} 
e^{-it^nH_0} \tau_{x^n}\psi(x)\|_{H^1}
\overset{ n\rightarrow \infty} \longrightarrow 0\end{equation*}
and since $\tau_{-x^n}$ are isometries on $H^1$ it is equivalent to
\begin{equation}\label{appext}
\|e^{-it^nH_q}\tau_{x^n}\psi(x)-
e^{-it^nH_0} \tau_{x^n}\psi(x)\|_{H^1}
\overset{ n\rightarrow \infty} \rightarrow 0, \hbox{ if } \quad t^n\overset{n\rightarrow \infty}
 \longrightarrow \bar t\in \R, \quad x^n \overset{ n\rightarrow \infty } \longrightarrow \pm \infty.\end{equation}
In order to prove \eqref{appext}
we use formula \eqref{ADelta}.
First notice that by a density argument we can assume $\psi\in C^{\infty}_0(\R)$.
In particular in the case 
$x^n\overset{n\rightarrow \infty} \longrightarrow -\infty$ we can assume 
$\tau_{x^n} \psi\subset (-\infty, 0)$
and in the case
$x^n\overset{n\rightarrow \infty}\longrightarrow +\infty$ we can assume 
$\tau_{x^n} \psi\subset (0, \infty)$.
In the first case we can combine \eqref{ADelta} and the translation invariance of the $H^1$ norm,
and hence \eqref{appext} becomes:
$$\|(e^{-i t^nH_0}(\tau_{x^n}\psi\star\rho_q))(x)\cdot 1_{x\geq 0}(x)+
(e^{-it^nH_0}(
\tau_{x^n}\psi\star\rho_q))(-x)\cdot 1_{x\leq 0}(x)\|_{H^1}\overset{
n\rightarrow \infty }\longrightarrow 0.
$$
Notice that $$(e^{-it^nH_0}(\tau_{x^n}\psi\star\rho_q))(x)
\cdot 1_{x\geq 0}(x)
=\tau_{x^n} \varphi(x) \cdot 1_{x\geq 0}(x)+ \tau_{x^n} r_n(x) \cdot 1_{x\geq 0}(x)$$
and  
$$(e^{-i t^n H_0}(
\tau_{x^n}\psi\star\rho_q))(-x)\cdot 1_{x\leq 0}(x)=(\tau_{x^n}
\varphi )(-x)\cdot 1_{x\leq 0}(x)+\tau_{x^n} r_n(x)\cdot 1_{x\leq 0}(x)$$
where $\varphi=e^{-i\bar t H_0}(\psi\star\rho_q)$
and $r_n(x)= (e^{-i t^n H_0} -e^{-i\bar t H_0}) (
\psi\star\rho_q)$. Notice that by continuity property of the flow
$e^{-itH_0}$ we get $r_n\overset{H^1}
 \longrightarrow 0$ and
since we are assuming $x^n\overset{ n\rightarrow \infty } \longrightarrow -\infty$, we get
$$
\|(\tau_{x^n} \varphi)(x) \cdot 1_{x\geq 0}(x)\|_{H^1}\overset{ n\rightarrow \infty }
 \longrightarrow 0, \quad \|(\tau_{x^n} \varphi)(-x)
\cdot 1_{x\leq 0}(x)\|_{H^1} \overset{ n\rightarrow \infty }
 \longrightarrow 0.$$
In the second case (i.e. $x^n\overset{ n\rightarrow \infty } \longrightarrow +\infty$) we get 
$$ e^{-it^nH_q} \tau_{x^n}\psi(x)=  R e^{-it^nH_q} R \tau_{x^n}\psi$$
where $Rf(x)=f(-x)$ and we used $R^2=Id$ and $[e^{-i t^n H_q}, R]=0$.
Hence \eqref{appext} follows since we have the following identity
$$e^{-it^nH_q} \tau_{x^n}\psi-e^{-it^nH_0} \tau_{x^n}\psi=  
R( e^{-it^nH_q} R \tau_{x^n} \psi - e^{-it^nH_0} R \tau_{x^n} \psi),$$
$R$ is an isometry in $H^1$ and moreover 
$R \tau_{x^n} \psi=\tau_{-x^n} R\psi(x)$ where $-x^n
\overset{ n\rightarrow \infty } \longrightarrow -\infty$.
Hence we are reduced to the previous case (i.e. $x^n\overset{ n\rightarrow \infty }\longrightarrow -\infty$).
\\
\\
Next we prove \eqref{abstr2}, i.e.
\begin{equation}\label{equile}
t^n\overset{n\rightarrow \infty}
 \longrightarrow \pm \infty \Longrightarrow \hbox{ (up to subsequence) } \|e^{it^nH_q} \tau_{x^n} \psi\|_{L^p} \overset{n\rightarrow \infty}\longrightarrow 0,
 \quad \forall \psi\in H^1,
\end{equation}
where $p\in (2, \infty)$.
We treat only the case $t^n\overset{n\rightarrow \infty}
 \longrightarrow +\infty$, the other case is equivalent.
By combining the time decay estimate $\|e^{-itH_q}\|_{{\mathcal L}(L^1, L^\infty)}
\leq C t^{-1/2}$ with the uniform bound $\|e^{-itH_q}\|_{{\mathcal L}(H^1, L^p)}\leq C$ and with a density argument, we deduce $
e^{itH_q}\overset{t\rightarrow \infty}\longrightarrow 0$ in the strong topology of the operators ${\mathcal L}(H^1, L^p)$.
Hence we conclude \eqref{equile} in the case
$x^n\overset{n\rightarrow \infty} \longrightarrow \bar x\in \R$
by a compactness argument.
Hence it is sufficient to prove
\begin{equation}\label{appextp}
\|e^{-it^nH_q}\tau_{x_n}\psi(x)-
e^{-it^nH_0} \tau_{x_n}\psi(x)\|_{L^p}
\overset{ n\rightarrow \infty } \longrightarrow 0, \hbox{ if } \quad t^n\overset{ n\rightarrow \infty }
\longrightarrow -\infty, \quad x^n\overset{ n\rightarrow \infty }
\longrightarrow \pm \infty,\end{equation}
and to conclude by the decay properties of the group $e^{itH_0}$.
The proof of \eqref{appextp} can be done via a density argument by assuming
$\psi\in C^\infty_0(\R)$
and $x^n\overset{n\rightarrow \infty}
\longrightarrow -\infty$ (the case $x^n\overset{n\rightarrow \infty}
\longrightarrow +\infty$ can be reduced to the previous one 
via the reflexion operator $R$, exactly as we did above along the proof
of \eqref{appext}).
Hence we can rely on \eqref{ADelta} and we are reduced to prove
 $$\|e^{-it^nH_0}(\tau_{x^n}\psi\star\rho_q))(x)
\cdot 1_{x\geq 0}(x)\|_{L^p}\overset{ n\rightarrow \infty } \longrightarrow 0,
$$$$ \|e^{-i t^n H_0}(
\tau_{x^n}\varphi\star\rho_q))(-x)\cdot 1_{x\leq 0}(x)\|_{L^p}
\overset{ n\rightarrow \infty } \longrightarrow 0$$
which is equivalent to 
$$\|\tau_{x^n} (e^{-it^nH_0}(\psi\star\rho_q))(x)\cdot
1_{x\geq 0}(x)\|_{L^p}\overset{ n
\rightarrow \infty} \longrightarrow 0,$$$$ \|\tau_{x^n} (e^{-i t^n H_0}(
\varphi\star\rho_q))(-x)\cdot 1_{x\leq 0}(x)\|_{L^p}\overset{ n
\rightarrow \infty} \longrightarrow 0.$$
Notice that the facts above follow by the translation invariance of the $L^p$ norm and by the property
$
e^{-itH_0}\overset{t\rightarrow -\infty}\longrightarrow 0$ in the strong topology of the operators ${\mathcal L}(H^1, L^p)$.\

\section{Preliminary results}

Since now on we shall use the following notations:
$$r=\alpha+2, \quad p=\frac{2\alpha(\alpha+2)}{\alpha+4}, \quad q=\frac{2\alpha(\alpha+2)}{\alpha^2 - \alpha -4},
$$
and $\alpha>4$ is the same parameter as in \eqref{NLS}. 
\subsection{Strichartz estimates}

We recall the following homogeneous and 
inhomogeneous Strichartz estimates:
\begin{align}
\|e^{-itH_q}\varphi\|_{L^pL^r}&\leq C \|\varphi\|_{H^1};
\\
\|e^{-itH_q}\varphi\|_{L^\alpha L^\infty}&\leq C \|\varphi\|_{H^1};
\\
\|\int_0^t e^{-i(t-s)H_q} F(s) ds\|_{L^p L^r}
&\leq C \|F\|_{L^{q'}L^{r'}};
\\
\|\int_0^t e^{-i(t-s)H_q} F(s) ds\|_{L^\alpha L^\infty}
&\leq C \|F\|_{L^{q'}L^{r'}}.
\end{align}
Since $p,r>6$, the first one is obtained by $\|e^{-itH_q}\varphi\|_{L^pL^{2p/(p-4)}}
\leq C\|\varphi\|_{L^2}$ in conjunction with a Sobolev embedding.
Since $\alpha>4$, the second one is obtained by interpolating between the 1-d admissible space $L^4L^\infty$ and $L^\infty L^\infty$. The third one enters the frame of non-admissible inhomogeneous Strichartz estimates in Lemma 2.1 in Cazenave-Weissler \cite{CaWe92}. The last one is contained in Theorem 1.4 of Foschi \cite{Fo05}, who extends the  non-admissible inhomogeneous Strichartz exponents.

\subsection{Perturbative nonlinear results}

\begin{prop}\label{giube}
Let $\varphi\in H^1$ be given and assume that the unique global solution to \eqref{NLS} 
$u(t, x)\in {\mathcal C} H^1$ satisfies $u(t, x) \in L^pL^r$.
Then $Sc(\varphi)$ occurs. 
\end{prop}

\begin{proof}
We first prove that
$$u(t, x)\in L^\infty H^1\cap L^pL^r \Longrightarrow 
u(t, x)\in L^\alpha L^\infty.
$$
It follows by the following chain of inequalities
$$\|u\|_{L^\alpha  L^\infty}\leq C( \|\varphi\|_{H^1}
+ \|u|u|^\alpha\|_{L^{q'} L^{r'}} ) \leq C( 
\|\varphi\|_{H^1}
+ \|u\|_{L^pL^r}^{\alpha+1})<\infty,
$$
where we have used the Strichartz estimates above.
We shall exploit also the following trivial estimate:
$$
\|\int_{t_1}^{t_2} e^{isH_q} F(s) ds\|_{H^1}
\leq \|F(s)\|_{L^1((t_1, t_2); H^1)}, \quad \forall t_1, t_2.
$$
Hence by the integral equation
$$\|e^{it_1H_q} u(t_1,.)- e^{it_2H_q} u(t_2,.)\|_{H^1}=
\|\int_{t_1}^{t_2} e^{isH_q} (u(s)|u(s)|^\alpha) ds\|_{H^1}
$$$$\leq \|u(s)|u(s)|^\alpha\|_{L^1((t_1,t_2);H^1)}
\leq C \|u\|_{L^\infty H^1} 
\|u\|_{L^\alpha_{(t_1,t_2)}  L^\infty}^\alpha
\overset{t_1, t_2\rightarrow \pm \infty} \longrightarrow
0.$$
Hence we get scattering via a standard argument.


\end{proof}

\begin{prop}\label{giuti}
There exists $\epsilon_0>0$ such that : 
$$\varphi\in H^1, \quad \|\varphi\|_{H^1}<\epsilon_0
\Longrightarrow \|u\|_{L^pL^r}\leq C_{\epsilon_0}\|\varphi\|_{H^1}, \quad \|v\|_{L^pL^r}\leq C_{\epsilon_0}\|\varphi\|_{H^1},$$
where $u,v$ are the solutions of \eqref{NLS} and \eqref{NLSconstant} respectively.  
\end{prop}

\begin{proof}
It is sufficient to check that
if $u(t, x)\in \mathcal{C} H^1$ is the unique global solution to \eqref{NLS},
then $u(t, x)\in L^pL^q$. In fact by the Strichartz estimates we get
$$\|u\|_{L^p((-T, T);L^r)}\leq C( \|\varphi\|_{H^1} +
\|u|u|^\alpha\|_{L^{q'}((-T, T);L^{r'})})
\leq C( \epsilon + \|u\|_{L^p((-T, T); L^r)}^{\alpha+1}).
$$
We conclude by a continuity argument that if $\epsilon$ is small enough, then
$\sup_T \|u\|_{L^p((-T, T); L^r)}<\infty$ and hence $u\in L^pL^r$. The proof goes the same for $v$.

\end{proof}


We also need the following perturbation result.
\begin{prop}\label{pertub}
For every $M>0$ there exists $\epsilon=\epsilon(M)>0$ and $C=C(M)>0$ 
such that the following occurs.
Let
$v\in {\mathcal C} H^1 \cap L^pL^r$ 
be a solution of the integral equation with source term $e(t,x)$: 
$$v(t,x)= e^{-itH_q} \varphi - i\int_0^t e^{-i(t-s)H_q} (v(s)|v(s)|^\alpha)(x) ds
+e(t,x)
$$
with $\|v\|_{L^p  L^r}<M$ and 
$\|e\|_{L^p  L^r}<\epsilon$. Assume moreover that
$\varphi_0 \in H^1$ is such that $\|e^{-itH_q }\varphi_0\|_{L^p  L^r}<\epsilon$,
then the solution $u(t, x)$ to \eqref{NLS} with initial condition $\varphi+\varphi_0$:
$$u(t,x)= e^{-itH_q} (\varphi+\varphi_0) - i\int_0^t e^{-i(t-s)H_q} (u(s)|u(s)|^\alpha)ds,
$$
satisfies $u\in L^p  L^r$ and moreover $\|u-v\|_{L^p L^r}<C\epsilon$.
\end{prop}

\begin{proof} 
It is contained in \cite{FaXiCa11}, setting the space dimension $N=1$. 

\end{proof}

\subsection{The nonlinear profiles}

\begin{prop}\label{sfugg}
Let $(x^n)_{n\in \N}$ be a sequence of real numbers 
such that $|x^n|\rightarrow +\infty$, $\psi\in H^1$
and $U(t, x)\in {\mathcal C} H^1\cap L^pL^r$ be the unique solution to \eqref{NLSconstant}
with initial data $\psi$.
Then we have 
$$U_n(t, x)=e^{-itH_q} \psi_n + i\int_0^t e^{-i(t-s)H_q} (U_n(s)|U_n(s)|^\alpha) ds +
g_n(t,x)$$
where
$$\psi_n=\tau_{x^n}\psi \hbox{ and } U_n=U(x-x^n, t)$$
and
$$
\|g_n(t,x)\|_{L^p L^r}\overset {n\rightarrow \infty} \longrightarrow 0.$$
\end{prop}

\begin{proof}
We are reduced to show:
\begin{align}\label{formfin1} \|e^{-itH_q}\psi_n -  e^{-itH_0}
\psi_n\|_{L^pL^r} &\overset {n\rightarrow \infty} \longrightarrow 0;\\
\label{formfin2} \|\int_0^t e^{-i(t-s)H_q} (U_n(s)|U_n(s)|^\alpha) ds 
&- \int_0^t e^{-i(t-s)H_0} (U_n(s)|U_n(s)|^\alpha) ds
\|_{L^pL^r} \overset {n\rightarrow \infty} \longrightarrow 0.
\end{align}
First we prove \eqref{formfin1} via
the formula \eqref{ADelta}.
Notice that by a density argument we can assume 
that $\psi$ is compactly supported.
Moreover modulo subsequence we can assume 
$x^n\overset{ n\rightarrow \infty } \longrightarrow \pm \infty$. In the case
$x^n\overset{ n\rightarrow \infty } \longrightarrow -\infty$ we get 
$\hbox{ supp } (\tau_{x^n} \varphi)\subset (-\infty, 0)$
and in the case
$x_n\overset{n\rightarrow \infty} \rightarrow +\infty$ we get
$\hbox{ supp } \tau_{x_n} \varphi\subset (0, \infty)$, for large $n$.
In the first case we can use \eqref{ADelta}
and hence \eqref{formfin1} becomes:
$$\|(e^{-i tH_0}(\tau_{x^n}\psi\star\rho_q))(x)\cdot 1_{x\geq 0}(x)+
(e^{-it H_0}(
\tau_{x^n}\psi\star\rho_q))(-x)\cdot 1_{x\leq 0}(x)
\|_{L^p L^r} \overset{n\rightarrow \infty}\longrightarrow 0.
$$
Notice that $$(e^{-itH_0}(\tau_{x^n}\psi\star\rho_q))(x)\cdot 
1_{x\geq 0}(x)
=\tau_{x^n} (e^{-itH_0} (\psi
\star\rho_q)) \cdot 1_{x\geq 0}(x)$$
and  
$$(e^{-i t H_0}(
\tau_{x^n}\psi\star\rho_q))(-x)\cdot 1_{x\leq 0}(x)=(\tau_{x^n} e^{-i t H_0}
(\psi\star\rho_q) )(-x)\cdot 1_{x\leq 0}(x)$$ 
hence we conclude
since by the usual Strichartz estimate we have
$e^{-i t H_0} (\psi\star\rho_q) \in L^p L^r$ 
and moreover we are assuming $x^n\overset{ n\rightarrow \infty } \longrightarrow -\infty$.\\
In the case $x^n\overset{ n\rightarrow \infty } \longrightarrow +\infty$
we can reduce to the case $x^n\overset{ n\rightarrow \infty } \longrightarrow -\infty$ via the reflection
operator $R$ (see the proof of \eqref{appext}).\\
\\
Next we focus on the proof of \eqref{formfin2}.
As above we shall assume $x^n\overset{ n\rightarrow \infty }
\longrightarrow -\infty$ (the other case 
$x^n\overset{ n\rightarrow \infty } \longrightarrow +\infty$ can be handled
via the reflection operator $R$).
We shall prove the following fact:
$$\|\int_0^t e^{-i(t-s)H_q} F(s, x-x^n) ds 
- \int_0^t e^{-i(t-s)H_0} F(s, x-x^n) ds\|_{L^p L^r}
\overset{n\rightarrow \infty} \longrightarrow 0, \quad \forall F(t,x)\in L^1 H^1.
$$
Notice that this fact is sufficient to conclude
since  by the classical scattering theory for NLS with 
constant coefficients we have $U|U|^\alpha\in L^1H^1$.

By a density argument we can assume 
the existence of a compact $K\subset \R$ such that
$\hbox { supp } F(t, x)\subset K, \quad \forall t$.
In particular for $n$ large enough we get
$\hbox{ supp } F(t, x-x^n) \subset (-\infty, 0), \quad \forall t.$
Hence we can use formula \eqref{ADelta} and we are reduce to prove
\begin{align*}&\|\big( \int_0^t e^{-i (t-s)H_0}(\tau_{x^n}F(s) \star\rho_q)(x)
ds\big) \cdot  1_{x\geq 0}(x)\|_{L^pL^r} \overset{n\rightarrow \infty}\longrightarrow 0;
\\\nonumber 
 &\|\big(\int_0^t
e^{-i(t-s) H_0}(
\tau_{x^n}F(s) \star\rho_q) (-x) ds\big) \cdot1_{x\leq 0}(x) \|_{L^pL^r}
\overset{n\rightarrow \infty}\longrightarrow 0.
\end{align*}
Next notice that
\begin{align*}\big( \int_0^t e^{-i (t-s)H_0}&(\tau_{x^n}F(s) \star\rho_q)(x)
ds\big) \cdot 1_{x\geq 0}(x)\\\nonumber&= \tau_{x^n} \big( \int_0^t (e^{-i (t-s)H_0}(F(s) \star\rho_q))(x)
ds\big) \cdot  1_{x\geq 0}(x)
\end{align*}
and
\begin{align*}
\big( \int_0^t
e^{-i(t-s) H_0}
&(
\tau_{x_n} F(s) \star\rho_q))(-x) ds\big)\cdot1_{x\leq 0}(x) 
\\\nonumber
&=\tau_{x^n} 
\big( \int_0^t
(e^{-i(t-s) H_0}(
F(s) \star\rho_q))ds\big)(-x) \cdot 1_{x\leq 0}(x) .
\end{align*}
Since $F(t,x) \star \rho_q\in L^1H^1$  
we get by Strichartz estimates
$\int_0^t
e^{-i(t-s) H_0}(
F(s) \star\rho_q)) ds\in L^p L^r$.
We conclude since 
$x^n\overset{n\rightarrow \infty} \longrightarrow -\infty$.

\end{proof}

\begin{prop}\label{nonlprof}
Let $\varphi\in H^1$, then there exist $W_\pm\in \mathcal {C}H^1\cap L^p_{\R^\pm}L^r$ solution to \eqref{NLS} and 
such that
\begin{equation}\label{ipso}\|W_\pm(t,.) - e^{-itH_q}\varphi\|_{H^1}\overset{t\rightarrow \pm \infty} \longrightarrow 0.\end{equation}
Moreover, if $(t^n)_{n\in \N}\subset \R$ is such that $t^n\overset{ n\rightarrow \infty }
\longrightarrow \mp \infty$, then
$$W_{\pm,n}(t,x)= e^{-it H_q} \varphi_n- i\int_0^t e^{-i(t-s)H_q} (W_{\pm,n}(s) |W_{\pm,n}
(s)|^\alpha) ds
+f_{\pm, n}(t,x)$$
where 
$$\varphi_n=e^{i t^n H_q} \varphi \hbox{ and } W_{\pm, n}(t,x)=
W_\pm (t-t^n, x)$$
and
$$\|f_{\pm,n}(t,x)\|_{L^p L^r}\overset{ n\rightarrow \infty }\longrightarrow 0.$$
\end{prop}



\begin{proof}The first part of the statement concerning the existence of wave operators is classical, since $e^{-it H_q}$ enjoys Strichartz estimates as $e^{-it H_0}$ and since we are in the defocusing case insuring global existence. 
For the second part of the statement we notice that by the translation invariance with respect to time we get
$f_{\pm , n}(t, x)= e^{-it H_q} (W_\pm (-t^n) -  \varphi_n)$. We conclude by combining Strichartz estimates with \eqref{ipso}.
\end{proof}

\begin{prop}\label{minusinf}
Let $(t^n)_{n\in \N}, (x^n)_{n\in \N}$ be sequences of numbers
such that $t^n\overset{ n\rightarrow \infty }
\longrightarrow \mp \infty$ and $|x^n|\overset{ n\rightarrow \infty } \longrightarrow + \infty$,
$\varphi\in H^1$ and $V_\pm (t, x)\in {\mathcal C} H^1\cap L^pL^r$ be a solution to \eqref{NLSconstant} such that
\begin{equation}\label{asympt}\|V_\pm (t,.)-e^{-itH_0} \varphi
\|_{H^1} \overset{ t\longrightarrow \pm \infty} \longrightarrow 0.\end{equation}
Then we have
$$V_{\pm,n}(t,x)= e^{-it H_q} \varphi_n-i \int_0^t e^{-i(t-s)H_q} (V_{\pm,n}(s) |V_{\pm,n}
(s)|^\alpha) ds
+e_{\pm, n}(t,x)$$
where 
$$\varphi_n=e^{i t^n H_q} \tau_{x_n} \varphi \hbox{ and } V_{\pm, n}(t,x)=
V_\pm (t-t^n, x-x^n)$$
and
$$\|e_{\pm,n}(t,x)\|_{L^p L^r}\overset{ n\rightarrow \infty }\longrightarrow 0.$$
\end{prop}

\begin{proof}
By combining \eqref{asympt} with the integral equation solved by $V_\pm(t, x)$,
it is sufficient to prove: 
$$\|e^{-i (t-t^n) H_q} \tau_{x^n} \varphi -e^{-i (t-t^n) H_0} 
\tau_{x^n}\varphi\|_{L^pL^r}
\overset{ n\rightarrow \infty}
\longrightarrow 0;$$
$$\|\int_0^t e^{-i(t-s)H_q} V_{\pm,n}(s) |V_{\pm,n}(s)|^\alpha ds
-\int_0^t e^{-i(t-s)H_0} V_{\pm,n}(s) |V_n(s)|^\alpha ds\|_{L^pL^r}
\overset{ n\rightarrow \infty } \longrightarrow 0.$$
The first one reduces to  \eqref{formfin1} by the change of variable $t-t^n\rightarrow t$. 
By the same change of variable the second one reduces to the estimate in \eqref{formfin2} with integral between $-t^n$ and $t$; its proof is similar to the one of \eqref{formfin2}.

\end{proof}

\section{Proof of Theorem \ref{main}}

In this section we prove the scattering result in Theorem \ref{main}. For this aim
we introduce the critical energy level defined as follows:
$$E_c=\sup \{ E>0\quad |\quad \forall \varphi\in H^1, \quad E(\varphi)< E
\Longrightarrow u(t,x)\in L^pL^r\},$$
where $u(t,x)$ denotes the unique solution to \eqref{NLS} with initial data $\varphi$.
Our aim is to show that $E_c=+\infty$, then we can conclude by Proposition \ref{giube}.
Notice also that due to Proposition \ref{giuti} we have $E_c>0$.\\
The main strategy is to prove that
if by the absurd $E_c<\infty$, then $E_c$ is achieved by a suitable
critical initial data $\varphi_c\in H^1$ whose corresponding solution
that does not scatter and moreover enjoys suitable compactness properties. The existence of such an object will be excluded via a rigidity argument by Proposition \ref{proprigid}. Therefore we shall conclude that $E_c=+\infty$. \\

\subsection{Existence and compactness of a minimal element.}

 \begin{prop}\label{princprop}
Assume that $E_c<+\infty$, then there exists a non trivial initial data $\varphi_c\in H^1$
such that  the corresponding solution $u_c(t,x)$ to \eqref{NLS}  has the property that $\{u_c(t,x), t\in \R\} $ is relatively compact in $H^1$. 
\end{prop}
  
\begin{proof}
Since we are assuming $E_c<+\infty$ then we can select a sequence
$\varphi_{n}\in H^1$ such that $E(\varphi_{n})\overset{ n\rightarrow \infty}
\longrightarrow E_c$ and 
$u_n(t,x)\notin L^pL^r$ where $u_n(t,x)$ is the corresponding solution
to \eqref{NLS}. 
First we shall prove that under these hypotheses there exists a subsequence converging in $H^1$ to a function with energy $E_c$, whose nonlinear evolution by \eqref{NLS} does not scatter. For this purpose we use the profile decomposition for the $H^1$ uniformly bounded sequence $\varphi_{n}$:
\begin{equation}\label{decompbas}
\varphi_{n}=\sum_{j=1}^{J} e^{it_{j}^nH_q} \tau_{x_{j}^n}\psi_j + R_n^J
\end{equation}
where $\psi_1,..., \psi_J\in H^1$.
We fix $J$ large enough in a sense to be specified later. 
From the energy estimate \eqref{orthogenergy} we recall that
\begin{equation}\label{lll}E_c\geq  
\limsup_{n\rightarrow \infty}\sum_{j=1}^J E( e^{it_j^n H_q}\tau_{x_{j}^n} \psi_j).
 \end{equation}
Notice that in view of \eqref{localiz}, modulo rearrangement we can choose 
$0\leq J'\leq J^{''}\leq J^{'''}\leq J^{iv}\leq J$
such that:
\begin{align*} 
(t_{j}^n, x_j^n)=(0,0), \quad  \forall n,  & \quad   1\leq j\leq J',
\\\nonumber
t_{j}^n=0, \quad \forall n & \hbox{ and } \quad |x_j^n|\overset{ n\rightarrow \infty}
\longrightarrow \infty, \quad  J'+1\leq j\leq J''
\\\nonumber x_{j}^n=0, \quad \forall n &\hbox{ and }  \lim_{n\rightarrow \infty} t_{j}^n=+\infty, \quad  J''+1\leq j\leq J''''
\\\nonumber x_{j}^n=0, \quad \forall n &\hbox{ and } \lim_{n\rightarrow \infty} t_{j}^n=-\infty, \quad  J'''+1\leq j\leq J^{iv}
\\\nonumber \lim_{n\rightarrow \infty} |x_{j}^n|=+\infty &\hbox{ and } \lim_{n\rightarrow \infty} t_{j}^n=+\infty, \quad  J^{iv}+1\leq j\leq J^{v}
\\\nonumber\lim_{n\rightarrow \infty} |x_{j}^n|=+\infty &\hbox{ and } \lim_{n\rightarrow \infty} t_{j}^n=-\infty, \quad  J^{v}+1\leq j\leq J.
\end{align*}
Above we are assuming that if $a>b$ then there is no $j$ such that $a\leq j\leq b$. 
Notice that by the condition \eqref{orth} we have that $J'\in \{0, 1\}$.

Next we shall prove that in \eqref{decompbas} we have $J=1$
and the remainder can be assumed arbitrary small in Strichartz norm. To this purpose we shall suppose by absurd that $J>1$ and we can  consider two cases:\\
- $J'=1$; \\
- $J'=0$.
\\
We shall treat only the first case which is the most complicated one - the other case is a simplified version of the first case. 
Then we have $(t_1^n,x_1^n)=(0,0)$
and we also have (recall that we are assuming by the absurd $J>1$) 
by \eqref{lll} that 
$E(\psi_1)<E_c$.  
Hence by definition of $E_c$ we get the existence of $N(t,x)\in {\mathcal C}H^1\cap L^pL^r$ such that
$$N(t,x)=e^{-itH_q} \psi_1 -i\int_0^t e^{-i(t-s)H_q} (N(s) |N(s)|^\alpha) ds.$$
\noindent
For every $j$ such that $J'+1\leq j\leq J''$ 
we associate with the profile $\psi_j$ the function $U_j(t, x)\in {\mathcal C}H^1
\cap L^pL^r$ according with Proposition \ref{sfugg}.
In particular we introduce $U_{j, n}= U_j(t, x-x^n_j)$.
\\
\\
For every $j$ such that $J''+1\leq j\leq J'''$ 
we associate with the profile $\psi_j$ the function $W_{-,j}(t, x)\in {\mathcal C}H^1
\cap L^p_{\R^-}L^r$ according with Proposition \ref{nonlprof}.
We claim that $W_{-,j}(t, x)\in {\mathcal C}H^1 \cap L^p L^r$. In fact by \eqref{lll} we get $E(e^{t^n_j H_q} \psi_j)<E_c-\frac{\|\nabla\psi_{j'}\|_{L^2}^{2}}{4}$ for some $1\leq j'\neq j\leq J$, whose existence is insured by the hypothesis $J>1$.  Hence $E(W_{-,j}(t, x))= \lim_{n\rightarrow \infty } E(e^{t^n_j H_q} \psi_j)<E_c$, so $W_{-,j}$ scatters both forward and backwards in time and therefore $W_{-,j}(t, x)\in {\mathcal C}H^1 \cap L^p L^r$.
In the sequel we shall denote $W_{-,j, n}= W_{-,j}(t-t^n_j, x)$.
\\\\
For every $j$ such that $J'''+1\leq j\leq J^{iv}$ 
we introduce in a similar way following Proposition \ref{nonlprof} the nonlinear solutions 
$W_{+,j}(t, x)\in {\mathcal C}H^1\cap L^p L^r$ and also 
$W_{+,j, n}= W_{+,j}(t-t^n_j, x)$.
\\\\
For every $j$ such that $J^{iv}+1\leq j\leq J^v$ 
we associate with $\psi_j$  the function $V_{-,j}(t, x)\in {\mathcal C}H^1
\cap L^pL^r$ according with Proposition \ref{minusinf}
and also 
$V_{-,j, n}= V_{-,j}(t-t^n_j, x-x^n_j)$.
\\\\
Finally, for every $j$ such that $J^v+1\leq j\leq J$ 
we associate with $\psi_j$  the function $V_{+,j}(t, x)\in {\mathcal C}H^1
\cap L^pL^r$ according with Proposition \ref{minusinf}
and also 
$V_{+,j, n}= V_{+,j}(t-t^n_j, x-x^n_j)$.\\

Our aim is to apply the perturbative result of Proposition \ref{pertub} to $u_n$ and to $Z_{J,n}$ defined as follows:
$$Z_{J,n}= N+ \sum_{j=J'+1}^{J''} U_{j,n}
+\sum_{j=J''+1}^{J'''} W_{-,j,n} +\sum_{j=J'''+1}^{J^{iv}} W_{+,j,n}+
\sum_{j=J^{iv}+1}^{J^{v}} V_{-,j,n}+\sum_{j=J^{v}+1}^{J} V_{+,j,n}.
$$
Notice that by combining Propositions \ref{sfugg}, \ref{nonlprof} and \ref{minusinf}
the function $Z_{J,n}$ satisfies:
$$Z_{J,n}(t)= e^{-itH_q} (\varphi_n - R^J_n)
- i z_{J,n}
+ r_{J,n},
$$
where $\|r_{J,n}\|_{L^pL^q}\overset{n\rightarrow \infty}\longrightarrow 0$ and 
$$z_{J,n}(t, x)=\int_0^t e^{-i(t-s)H_q} ( N(s)|N(s)|^\alpha ) ds 
+ \sum_{j=J'+1}^{J''} \int_0^t e^{-i(t-s)H_q} ( U_{j,n}(s)|U_{j,n}(s)|^\alpha ) ds $$
$$+ \sum_{j=J''+1}^{J'''} \int_0^t e^{-i(t-s)H_q} ( W_{-,j,n}(s)|W_{-,j,n}(s)|^\alpha ) ds + \sum_{j=J'''+1}^{J^{iv}} \int_0^t e^{-i(t-s)H_q} ( W_{+,j,n}(s)|W_{+,j,n}(s)|^\alpha ) ds$$
$$+ \sum_{j=J^{iv}+1}^{J^v} \int_0^t e^{-i(t-s)H_q} ( V_{-,j,n}(s)|V_{-,j,n}(s)|^\alpha ) ds
+ \sum_{j=J^{v}+1}^{J} \int_0^t e^{-i(t-s)H_q} ( V_{+,j,n}(s)|V_{+,j,n}(s)|^\alpha ) ds.$$
We note that in Lemma 6.3 of \cite{FaXiCa11}, the estimates in $L^pL^r$ needed to apply the perturbative result are proved first in a space $L^\gamma L^\gamma$, and then concluded by an uniform bound in $L^\infty H^1$ of the approximate solutions. This last uniform bound, proved in Corollary 4.4 of \cite{FaXiCa11}, is more delicate in our case. Therefore we prove in the appendix estimates directly for the $L^pL^r$ norm. 
In view of Corollary \ref{appencor} in the appendix, based  on the orthogonality condition \eqref{orth}, we have
$$\|z_{J,n}(t,x)- \int_0^t e^{-i(t-s)H_q} (Z_{J, n} (s)|Z_{J,n}(s)|^\alpha ) ds\|_{L^pL^r} \overset{n\rightarrow \infty} \longrightarrow 0. 
$$
Summarizing we get:
$$Z_{J,n}(t)= e^{-itH_q} (\varphi_n - R^J_n)
- i\int_0^t e^{-i(t-s)H_q} (Z_{J, n} (s)|Z_{J,n}(s)|^\alpha ) ds + s_{J,n}$$
with $\|s_{J,n}\|_{L^pL^r}\overset{n\rightarrow \infty} \longrightarrow 0
$.
In order to apply the perturbative result of Proposition \ref{pertub}, we need also a bound on $\sup_J (\limsup_{n\rightarrow \infty} \|Z_{J,n}\|_{L^pL^r})$. Corollary \ref{appencorbis} ensures us that
$$\limsup_{n\rightarrow \infty} (\|Z_{J,n}\|_{L^pL^r})^{1+\alpha}\leq 2\|N\|_{L^pL^r}^{1+\alpha}+ 2\sum_{j=J'+1}^{J''} \|U_{j}\|_{L^pL^r}^{1+\alpha}$$
$$+2\sum_{j=J''+1}^{J'''} \|W_{-,j}\|_{L^pL^r}^{1+\alpha} +2\sum_{j=J'''+1}^{J^{iv}} \|W_{+,j}\|_{L^pL^r}^{1+\alpha}+
2\sum_{j=J^{iv}+1}^{J^{v}} \|V_{-,j}\|_{L^pL^r}^{1+\alpha}+2\sum_{j=J^{v}+1}^{J} \|V_{+,j}\|_{L^pL^r}^{1+\alpha}.$$
By using the defocusing conserved energy we obtain that the initial data of the wave operators $V_{\pm,j},W_{\pm,j}$ are upper-bounded in $H^1$ by $C\|\varphi_j\|_{H^1}$. In view of the orthogonality relation \eqref{orthog} we obtain the existence of $J_0$ such that for any $J\geq J_0$ we have
$$\|\varphi_j\|_{H^1}<\epsilon_0,$$
where $\epsilon_0$ is the universal constant in Proposition \ref{giuti}. Then by Proposition \ref{giuti}, the fact that $N,U_j,V_{\pm,j},W_{\pm,j}$ belong to $L^pL^r$ 
and by Corollary \ref{appencorbis} we get
\begin{equation}\label{imps}
\sup_J (\limsup_{n\rightarrow \infty} \|Z_{J,n}\|_{L^pL^r})=M<\infty.
\end{equation}
Due to \eqref{imps}
we are in position to apply Proposition \ref{pertub} to $Z_{J,n}$ 
provided that we choose $J$ large enough in such a way that
$\limsup_{n\rightarrow \infty} 
\|e^{-itH_q} R_n^J\|<\epsilon$, where $\epsilon=\epsilon(M)>0$ is the one given
in Proposition \ref{pertub}.
As a by-product we get that 
$Sc(\varphi_n)$ occurs for $n$ large, and hence we get a contradiction.
\\
\\
Therefore we have obtained that $J=1$ so 
\begin{equation}\label{decompbasone}
\varphi _{n}=e^{it_{1}^n H_q} \tau_{x_{1}^n}\psi_1 + R_n^1
\end{equation}
where $\psi_1\in H^1$ and 
$\limsup_{n\rightarrow \infty} \|e^{-it H_q} R_n^1\|_{L^p L^r}=0
$. Following the same argument as in Lemma 6.3 of \cite{FaXiCa11} 
one can deduce that $(t_1^n)_{n\in \N}$ is bounded and hence we can assume
$t_1^n=0$. Moreover arguing by the absurd and by combining Propositions \ref{sfugg}
and \ref{pertub} we get $x_1^n=0$ (otherwise $Sc(\varphi_n)$ occurs for $n$ large enough, and it is a contradiction).
We obtain then that $Sc(\psi_1)$ does not occurs, so 
$E(\psi_1)\geq E_c$. Equality occurs by the energy estimate \eqref{lll}, and in particular there is a subsequence of $(\varphi _{n})_{n\in \N}$ converging in $H^1$ to $\psi_1$. We have as a critical
element $\varphi_c=\psi_1$.
\\
\\
The compactness of the trajectory $u_c(t,x)\in H^1$
follows again by standard arguments. More precisely, for $(t^n)_{n\in \N}$ a sequence of times, $(u_c(t^n,x))_{n\in \N}$ satisfies the same hypothesis as $\varphi_n$ at the beggining of the proof above so we conclude that  there is a subsequence converging in $H^1$. 

\end{proof}

\subsection{Rigidity of compact solutions.}

We shall get now a constraint on the solution $u_c(t,x)$ constructed above.

\begin{prop}\label{proprigid}
Assume $u$ solves \eqref{NLS} with $q\geq 0$ and satisfies
the property:
\begin{equation}\label{compact}\{u(t,x), t\in\mathbb R\} \hbox{ is compact in } H^1.\end{equation}
Then $u=0$.
\end{prop}

\begin{proof} 
We start with the following virial computation. 

\begin{lem}\label{propNM}
Let $u(t, x)\in {\mathcal C} H^1 $ be a global solution to \eqref{NLS} and $\lambda(x)$ a weight such that $\partial_x \lambda(0)=0$. Then
\begin{align}\label{virial}
\frac {d^2}{dt^2} \int \lambda |u|^2 dx=\frac {d}{dt} (\Im  \int \partial_x \lambda
\partial_x u  \bar u dx)
 \\\nonumber 
= \int \lambda'' |u'|^2 dx
- \frac 14 \int \lambda^{iv} |u|^2 dx
+q \lambda^{''}(0)  |u(t,0)|^2 +\frac \alpha{\alpha+2} 
\int \lambda'' |u|^{\alpha+2}dx.
\end{align}

\end{lem}

\begin{proof}

We shall use the notations $u'=\partial_x u$, $\lambda'=\partial_x \lambda$,
$u_t=\partial_t u$ .
We compute, for a weight in space $\lambda(x)$:
\begin{align*}\frac d{dt} \int \lambda |u|^2 dx= 2\Re \int \lambda u_t \bar u dx\\
=2\Re \int \lambda (\frac i 2 u'' \bar u  - i u \bar u|u|^\alpha)dx +2\Re \lambda(0) i |u(t,0)|^2
=- \Re \int i \lambda'  u' \bar u  dx.
\end{align*}
Next we compute, due to the previous identity:
\begin{align}\label{vir}\frac {d^2}{dt^2} \int \lambda |u|^2 dx=
\Im \int \lambda' (u' \bar u)_t dx
\end{align}
We get, by using integrations by parts and $\lambda'(0)=0$:
\begin{align}\label{vir2}
\Im \int \lambda' (u' \bar u)_t dx&=
\Im \int \lambda' u' \bar u_t dx+ \Im\int \lambda' u_t' \bar u dx
\\\nonumber
&=2\Im \int \lambda' u' \bar u_t dx
- \Im\int \lambda'' u_t \bar u dx 
\\\nonumber =-\Re \int \lambda' u'  \bar u'' dx &+ 2\Re\int \lambda' u'  \bar u|u|^\alpha dx 
 - \frac 12 \Re \int \lambda'' u''\bar u
dx + \Re \int \lambda'' (u|u|^\alpha)\bar udx+q \lambda^{''}(0)  |u(t,0)|^2.
\end{align}
Next notice that
$$-\Re \int \lambda' u'  \bar u'' dx
- \frac 12 \Re \int \lambda'' u''\bar u
dx$$$$=\frac 12 \int \lambda'' |u'|^2 dx+ 
+ \frac 12 \Re \int \lambda''' u'\bar u dx + \frac 12 \Re \int \lambda'' |u'|^2 dx
$$
$$=  \int \lambda'' |u'|^2 dx
- \frac 14 \int \lambda^{iv} |u|^2 dx
,$$
and
$$2\Re\int \lambda' u'  \bar u|u|^\alpha dx 
+ \Re \int \lambda'' (u|u|^\alpha)\bar udx
$$
$$= -\frac{2}{\alpha+2} \int \lambda'' (|u|^{\alpha +2})
+ \int \lambda'' |u|^{\alpha+2}dx= \frac \alpha{\alpha+2} 
\int \lambda'' |u|^{\alpha+2}dx.$$
We conclude by combining the computations above with \eqref{vir} and \eqref{vir2}. 

\end{proof}

We continue the proof of Proposition \ref{proprigid} and we assume by the
absurd the existence of a non-trivial solution $u(t,x)$ that satisfies \eqref{compact}.
We fix a cut-off $\chi$ vanishing outside $B(0,2)$ and equal to one on $B(0,1)$. Let $R>0$ to be chosen later. By using Lemma \ref{propNM} for $\lambda(x)=x^2\chi(\frac {|x|}R)$  then we get:
\begin{align*}
\frac {d}{dt} (\Im  \int (x^2\chi(\frac {|x|}R))' u' \bar u dx)
\\\nonumber
\geq  \int_{|x|<R} |u'|^2dx+\frac{\alpha}{\alpha+2}\int_{|x|<R} |u|^{\alpha+2} dx
-C\int_{|x|>R} (|u|^2+ |u'|^2+ |u|^{\alpha+2}) dx
\\\nonumber
\geq \delta-C\int_{|x|>R} (|u|^2+ |u'|^2+ |u|^{\alpha+2}) dx,
\end{align*}
for some $\delta>0$. Notice that the existence of a positive $\delta$ comes from
the fact that $u(t,x)$ is assumed to be non trivial
and moreover satisfies \eqref{compact}. 
By integrating from $0$ to $t$ and using Cauchy-Schwartz inequality,
then we obtain 
$$C(R)\|u\|_{L^\infty H^1}\geq t\delta-C\int_0^t\int_{|x|>R}
(|u|^2+ |u'|^2+ |u|^{\alpha+2}) dx.$$
By using again the compacteness hypothesis \eqref{compact}
then we get a contradiction as $t$ goes to infinity, provided $R>0$ is large enough. 

\end{proof}

As a conclusion, the existence of the solution $u_c(t,x)$ constructed in Proposition \ref{princprop} is constrained by  Proposition \ref{proprigid} to be the null function. Since $E(u_c)=E_c>0$ we get a constradiction, so the hypothesis $E_c<+\infty$ made in Proposition \ref{proprigid} cannot hold. Therefore we conclude that $E_c=+\infty$, so all solutions of \eqref{NLS} scatter. 

\section{Appendix}

\begin{prop}\label{aPPe}
Let $W_i(t, x)\in \mathcal CH^1\cap L^p L^r$ for $i=1,2$ 
be space-time functions and
$(t_n, s_n, x_n, y_n)_{n\in \N}$ be sequences of real numbers.
Assume that $|t_n-s_n|+ |x_n - y_n|\overset{n\rightarrow \infty} 
\longrightarrow +\infty$
then we get
$$\||W_1(t-t_n, x-x_n)|^\alpha \times |W_2(t-s_n, x-y_n)|\|_{L^{q'}L^{r'}}
\overset{n\rightarrow \infty} 
\longrightarrow 0.$$
\end{prop}

\begin{proof}
First assume that $|t_n-s_n|\overset{n\rightarrow \infty}
\longrightarrow \infty$.
Then in this case we get:
 $$\||W_1(t-t_n, x-x_n)|^\alpha \times |W_2(t-s_n, x-y_n)|\|_{L^{q'}L^{r'}}
 $$$$\leq \big \| \| W_1(t-t_n, x-x_n)\|_{L^r_x}^\alpha \times 
 \|W_2(t-s_n, x-y_n)\|_{L^r_x}\big  \|_{L^{q'}_t}.$$
The conclusion follows by the following elementary fact:
$$|t_n-s_n|\overset{n\rightarrow \infty}
\longrightarrow +\infty\Longrightarrow 
\||f_1(t-t_n)|^\alpha \times |f_2(t-s_n)|\|_{L^{q'}_t}
\overset{n\rightarrow \infty}
\longrightarrow 0,$$
where $f_i(t)=\|W_i(t, x)\|_{L^r}\in L^p_t$, $i=1,2$.\\
Next we assume that $(|t_n-s_n|)_{n\in \N}$ is bounded
and $|x_n-y_n|\overset{n\rightarrow \infty}
\longrightarrow \infty$.
First notice that
we have
$$\||W_1(t, x-x_n)|^\alpha \times |W_2(t+t_n -s_n, x-y_n)|\|_{L^{q'}_{|t|>T}L^{r'}}
$$$$\leq \|W_1(t,x)\| _{L^{p}_{|t|>T}L^{r}}^\alpha \|W_2(t+t_n-s_n,x)\| _{L^{p}_{|t|>T}L^{r}}
\overset{T\rightarrow \infty}
\longrightarrow 0.
$$
Hence it is sufficient to prove
\begin{equation}\label{LeBe}
\||W_1(t, x-x_n)|^\alpha \times |W_2(t+t_n -s_n, x-y_n)|\|_{L^{q'}_{|t|<T}L^{r'}}
\overset{n\rightarrow \infty}
\longrightarrow 0
\end{equation}
for every fixed $T$.
We notice that 
for every fixed $t$ we get:
\begin{align*}\||W_1(t, x-x_n)|^\alpha \times |W_2(t+t_n-s_n, x-y_n)|\|_{L^{r'}_x}
\\\nonumber=\||W_1(t, x)|^\alpha \times |W_2(t+t_n-s_n, x+x_n-y_n)|\|_{L^{r'}_x}
\overset{n\rightarrow \infty} \longrightarrow 0\end{align*}
where we used at the last step the following facts
(below we use the property $W_i(t, x)\in {\mathcal C} H^1$ to give a meaning to
the function 
$W_i(t, x)$ for every fixed $t$):
$$|W_1(t, x)|^\alpha\in L^{\frac r \alpha}, \quad \forall t$$
$$\{W_2(t +t_n-s_n, x), \quad n\in \N\} \hbox { is compact in } L^{r}, \quad \forall t$$
and $$|x_n-y_n|\overset{n\rightarrow \infty}
\longrightarrow \infty.$$
Indeed the first property above follows by the Sobolev embedding 
$H^1\subset L^r$, and second one follows from the fact that 
$(|t_n-s_n|)_{n\in \N}$ is bounded and the function
$\R \ni t\rightarrow W_2(t, x)\in H^1$ is continuous.
On the other hand we have
$$
\sup_{t\in (-T, T)}\||W_1(t, x)|^\alpha \times |W_2(t+t_n-s_n, x+x_n-y_n)|\|_{L^{r'}_x}
$$$$\leq \sup_{t\in (-T, T)} \|W_1(t, x)\|_{L^r_x}^\alpha 
\times \|W_2( (t+t_n-s_n, x)\|_{L^r_x}<\infty,
$$
where we used again the Sobolev embedding 
$H^1\subset L^r$ 
and the assumption $u(t, x)\in {\mathcal C}H^1$.
We deduce \eqref{LeBe} by the Lebesgue dominated convergence theorem.


\end{proof}

As a consequence we get the following corollary.
\begin{cor}\label{appencor}
Let $W_j(t, x)\in L^p L^r\cap {\mathcal C} H^1$, $j=1,..., N$ be a 
family of space-time functions and let
$(t_j^n, x_j^n)_{n\in \N}$, $j=1,...,N$ be sequences of real numbers that satisfy the ortogonality condition:
$$|t_j^n- t_k^n|+ |x_j^n- x_k^n|\overset{n\rightarrow \infty} \longrightarrow +\infty, \quad j\neq k.$$
Then we have
$$\|\sum_{j=1}^N W_{j, n}(t, x)
|W_{j, n}(t, x)|^\alpha  - (\sum_{j=1}^N W_{j, n}(t, x))
(|\sum_{j=1}^N W_{j, n}(t, x)|^\alpha )
 \|_{L^{q'}L^{r'}}\overset{n\rightarrow \infty} \longrightarrow 0,$$ where
$W_{j, n}(t, x)= W_j(t-t_j^n, x-x_j^n)$.
\end{cor}

\begin{proof}
It follows by Proposition \ref{aPPe} in conjunction
with the following elementary inequality
\begin{equation}\label{elementary}\big |\sum_{j=1}^N a_j |a_j|^\alpha -  (\sum_{j=1}^N a_j)
| \sum_{j=1}^N a_j|^\alpha \big |
\leq C(N,\alpha) \sum_{j\neq k} |a_j| |a_k|^\alpha, \quad \forall a_1,..., a_N\in \C.
\end{equation}
\end{proof}

\begin{cor}\label{appencorbis}
Let $W_j(t, x)\in L^p L^r\cap {\mathcal C} H^1$, $j=1,..., N$ be a 
family of space-time functions and let
$(t_j^n, x_j^n)_{n\in \N}$, $j=1,...,N$ be sequences of real numbers that satisfy the ortogonality condition:
$$|t_j^n- t_k^n|+ |x_j^n- x_k^n|\overset{n\rightarrow \infty} \longrightarrow +\infty, \quad j\neq k.$$
Then we have
$$\limsup_{n\rightarrow \infty}(\|\sum_{j=1}^N W_{j, n}(t, x)\|_{L^pL^r})^{1+\alpha} \leq 2\sum_{j=1}^N \|W_{j}\|_{L^p L^r}^{1+\alpha}$$ where
$W_{j, n}(t, x)= W_j(t-t_j^n, x-x_j^n)$. 
\end{cor}

\begin{proof}
We have
$$\|\sum_{j=1}^N W_{j, n}(t, x)\|_{L^pL^r}\leq 
(\| (\sum_{j=1}^N |W_{j, n}(t, x)|)^{1+\alpha} \|_{L^{q'}L^{r'}})^\frac 1 {1+\alpha}
$$
$$\leq (\| (\sum_{j=1}^N |W_{j, n}(t, x)|)^{1+\alpha} - 
\sum_{j=1}^N |W_{j, n}(t, x)|^{1+\alpha}\|_{L^{q'}L^{r'}}
+\|\sum_{j=1}^N |W_{j, n}(t, x)|^{1+\alpha}\|_{L^{q'}L^{r'}}
)^\frac 1{1+\alpha}.$$
The conclusion follows by combining \eqref{elementary} with Proposition \ref{aPPe}. 
\end{proof}

\bibliographystyle{alpha} 
\bibliography{ScatteringDelta}

\end{document}